\documentclass[letterpaper, 10pt, conference]{ieeeconf}

\usepackage{algorithmic}
\usepackage{algorithm}
\usepackage{amsmath}
\usepackage{amssymb}
\usepackage[T1]{fontenc}
\usepackage[latin1]{inputenc} 
\usepackage{color}

\newcommand{\MORPH}{\mathcal{S}}
\newcommand{\QNUM}{D}

\newcommand{\HYBINP}{\mathcal{U}}

\newcommand{\DNUM}{\mathcal{D}}

\newcommand{\LPV}{ALPV}
\newcommand{\BLPV}{ALPV\ }
\newcommand{\SLPV}{ALPVs}
\newcommand{\BSLPV}{ALPVs\ }

\newcommand{\GCR}{\textbf{GCR}}
\newcommand{\SysLPV}[1][{}]{ (\OUTP,m,n^{#1},\mathcal{P},\{(A_{q}^{#1},B_{q}^{#1},C_{q}^{#1})\}_{q=1}^{\QNUM})
}

\newcommand{\SPAN}{\mathrm{Span}}

\newcommand{\Rank}{\mathrm{rank}}

\newtheorem{Theorem}{Theorem}
\newtheorem{Lemma}{Lemma}

\newtheorem{Assumption}{Assumption}

\newtheorem{Definition}{Definition}

\newtheorem{Notation}{Notation}

\newcommand{\SWS}{\mathfrak{S}}
\newcommand{\SWI}{\mathfrak{I}}

\newcommand{\OUTP}{r}

\begin{document}
\title{Affine LPV systems: realization theory, input-output equations and relationship with linear switched systems}
\date{}
\author{Mih\'aly Petreczky$^\dag$,  
         Guillaume Merc\`ere$^{*}$ and 
        Roland T\'oth$^\ddag$  \\
       $^\dag$Ecole des Mines de Douai, F-59500 Douai, France, \\
      \texttt{mihaly.petreczky@mines-douai.fr} \\
        $^{*}$
        University of Poitiers, 
        Laboratoire d'Informatique et d'Automatique pour les Syst\`emes \\
        B.P. 633 86022 Poitiers Cedex, France \\
       \texttt{guillaume.mercere@univ-poitiers.fr} \\
       $^\ddag$ Control Systems Group, 
        Department of Electrical Engineering, Eindhoven University of Technology  \\
       P.O. Box 513, 5600 MB Eindhoven, The Netherlands \\
       \texttt{R.Toth@tue.nl}. 
}

\maketitle
\begin{abstract}
 We formulate a Kalman-style realization theory for
 discrete-time affine LPV systems.
 By an affine LPV system we mean an LPV system whose matrices
 are affine functions of the scheduling parameter. 
 In this paper we characterize those input-output behaviors which
 exactly correspond to affine LPV systems.
 In addition, we characterize minimal affine LPV systems which
 realize a given input-output behavior.
 Furthermore, we explain the relationship between Markov-parameters,
 Hankel-matrices, existence of an affine LPV realization and minimality.
 The results are derived by reducing the problem to the realization problem for linear switched systems.
 In this way, as a secondary contribution, we formally demonstrate
 the close relationship between LPV systems and linear switched systems.
 In addition we show that an input-output map
 has a realization by an affine 
 LPV system if and only if it satisfies
 certain types of input-output equations.
\end{abstract}

\section{Introduction}
 The paper presents a Kalman-style realization theory for
 discrete-time affine LPV systems.
 An affine LPV system (\emph{abbreviated by \LPV}) 
 is linear parameter-varying systems whose matrices are affine functions of the scheduling parameters.   
 By the input-output behavior of an \BLPV  we will mean 
 the input-output map induced by the zero initial state.
 The paper aims at answering the following
 questions.
 \begin{itemize}
 \item How can we characterize those input-output maps which can be
       described \BSLPV ? What is the role of
       Hankel-matrices in this characterization ?

 \item What can be said about minimal \BSLPV realizing
       the given input-output map ? What is the relationship
       between minimal \SLPV,
       and reachability and observability of such systems ?
       Are all minimal \BLPV realizations of the 
       same input-output map isomorphic ?

 \item How can we characterize the input-output equations solutions
       of which correspond to input-output maps of
       \BSLPV ?
 \end{itemize}
 In this paper we will show the following.
 \begin{itemize}
 \item
  We prove that reachability and observability of
   \BSLPV is equivalent to minimality and that minimal realizations of the same input-output map
   are isomorphic.   Note that isomorphism in this setting means
   a linear state-space transformation which does not depend on the
   scheduling parameter. 

   We also show that any \BLPV can be transformed
   into a minimal one while preserving its input-output map.
   
   In addition, we characterize reachability and
   observability in terms of rank conditions  for
   extended reachability and observability matrices.
 
 \item
   We define the 
  Markov-parameters as functions of the input-output map.
    We then show that
   the Hankel-matrix constructed from the Markov-parameters
   has a finite rank if and only   if the corresponding
   input-output map has a realization by an \LPV.
   We show that the Kalman-Ho 
   algorithm of \cite{RolandAbbas} can be used to
   compute an \BLPV realization from the Hankel-matrix, and
   we provide a bound on the size of the Hankel sub-matrices
   which guarantees correctness of the algorithm.
 
 \item
   We also present a class of input-output equations which
   characterize \BSLPV precisely: an input-output map is a 
   solution of such an input-output equation if and only if
   it admits a realization by an \BLPV.

 \item
   Finally, as a secondary result, we establish a formal
   equivalence between the realization problems for 
   \BSLPV and for linear switched systems.
    The solution of
   the latter problem is known
   \cite{MPLBJH:Real,MP:BigArticlePartI,MP:RealForm} and it is
    equivalent to that of recognizable formal power series and state-affine systems \cite{Reut:Book,Son:Real,MFliessHank}.
  We then use realization theory of linear switched systems to
  derive a Kalman-style realization theory for \SLPV.
 \end{itemize}


   Note that in this paper we consider \BSLPV with a fixed
   initial state. Just as in the linear switched case 
   \cite{MP:RealForm,MP:BigArticlePartI} it is possible
   to extend these results to the case of an arbitrary set of
   initial states.

\textbf{Motivation and novelty}
  To the best of our knowledge, the paper is new.
  Many of the concepts (Hankel-matrix, Markov-parameters, extended reachability/observability matrix, etc.) used in the paper 
 have already appeared before.
  However, what is truly novel in this paper is that it 
  formulates a Kalman-style realization theory for \SLPV, while
  using the existing concepts from the literature.
  In addition, the equivalence between \BLPV realizations and
  input-output equations is also new, to the best of our knowledge.

  A Kalman-like realization theory offers several benefits
  for system identification. It allows the characterization of
  identifiability and equivalence of state-space representations.   
   The latter is important for model validation. 
  Kalman-like realization theory also provides a tool for
  finding identifiable canonical parameterizations and characterizing the manifold structure of systems, including hybrid and nonlinear systems,
  \cite{MP:MTNSSpaces,MR0429197,GW74,RalfPeeters,HanzonBook,Hazewinkel1,MercereBako}.  In turn, this knowledge could be used for
  deriving new parametric identification algorithms, 
  see \cite{RalfPeeters,HanzonBook} for the linear case.
  Realization theory also leads to model reduction techniques, such
  as balanced truncation and moment matching \cite{ModRed1}.
  This is also true
  for linear switched systems \cite{MP:ADHS2012,MP:PartReal} and
  \BSLPV \cite{RolandAbbas}.

   Finally, the paper formulates the precise 
   relationship between the realization
   problems for \BSLPV and linear switched systems.
   While this relationship is part of the folklore, 
   it has not been stated formally yet. 

\textbf{Relationship with existing work}
   The field of identification of LPV systems is a mature one
   with a vast literature and several applications,
   without claiming completeness, we mention
   \cite{VerdultPhd,Verhaegen:Automatica,Verhaegen09,Verhaegen:INJC,Lal11,TothTac1,Toth1,THH09,Kal09,CCS11,CL08,BDG05,BG02,FWV07}.
   As it was mentioned before, many of the concepts 
   used in this paper were published before.
   In particular, the idea of Hankel-matrix appeared in \cite{RolandAbbas,VerdultPhd,Verhaegen:Automatica,Verhaegen09,Verhaegen:INJC}.
 However, \cite{RolandAbbas,VerdultPhd,Verhaegen:Automatica,Verhaegen09,Verhaegen:INJC} focuses on the identification problem, which is related to, but different from the realization problem studied in this paper.
  The Markov-parameters were already described in
  \cite{RolandAbbas,Verhaegen09}.
  In contrast to the existing work, in this paper the 
  Markov-parameters and Hankel-matrix are defined directly for 
  input-output maps, without assuming the existence of a finite dimensional \BLPV realization.  
  In fact, the finite rank of the Hankel-matrix represents the necessary and sufficient condition for the existence of an \BLPV realization.  
 The Kalman-Ho realization algorithm was discussed in \cite{RolandAbbas}, but it was formulated with the assumption that an \BLPV realization
  exists. Moreover, the conditions under which the algorithm 
  yields a true realization of the input-output map were not
  discussed in detail in \cite{RolandAbbas}.
  Extended observability and reachability matrices were presented in \cite{Verhaegen09,RolandAbbas}. However, their system-theoretic interpretation and relationship with minimality were not explored.

  Realization theory of more general linear parameter-varying
  systems was already developed in \cite{Toth1}.
  In \cite{Toth1} the system
  matrices are allowed 
  to depend on the scheduling parameter in a non-linear way.
  Moreover, in \cite{Toth1} no conditions involving the rank of the
  Hankel-matrix were formulated for the existence of a state-space 
  realization.
  Hence, the results of \cite{Toth1}
  do not always imply the ones presented in this paper. 
  The minimality conditions of \cite{Toth1} imply
  those of this paper. However, an \BLPV may be minimal in the sense
  of this paper, and may fail to be minimal in the sense of
  \cite{Toth1}. Intuitively this is not at all surprising, since
  it is conceivable that
  by allowing more complicated dependence on the scheduling parameter
  we can get rid of some states.

  In particular, minimal \BSLPV  in the sense of this paper
  are related by constant state-space isomorphism. This is in contrast
  to \cite{Toth1}, where the isomorphism relating state-space
  representations may depend on the scheduling parameter. 
  Note that a minimal \BSLPV
  in the sense of this paper need not be minimal in the sense of
  \cite{Toth1}. Hence, there might exist several state-space
  isomorphisms between \BSLPV which are minimal in the sense of this
  paper. Some of these isomorphisms might depend on the scheduling
  parameters. However, the results of this paper imply that there
  will be a constant state-space isomorphism.
  This is also consistent with \cite{TothKulcsar}.

  Although realization theory of \BSLPV is quite similar to that
  of linear switched systems, there are important differences.
  In particular, there exist no parallel for linear switched systems of
  the equivalence between realizability and existence of input-output
  equations. In fact,  \BSLPV seem to behave more like 
  state-affine systems \cite{Son:Resp,Son:Real} for which
  an analogous result exists.

  It is well known that there is a correspondence
  between LPVs and LFT representations \cite{VerdultPhd,Toth2}.
  In \cite{ball:1474,Beck1,Beck2} the theory of recognizable
  formal power series was used to develop
  realization theory for LFT representations.
  In this paper we reduce the realization
  problem of \BSLPV to that of for linear switched systems.
  The latter problem can also be solved by using 
  recognizable formal power series \cite{MP:RealForm,MP:BigArticlePartI,MPLBJH:Real}. Hence, there is an analogy
  between our approach and that of \cite{ball:1474,Beck1,Beck2}.
  Note that the transformations between \BSLPV and LFT representations
  involve non-trivial transformations of the system
  matrices. Moreover, the resulting class of LFT representations
  seem to differ from the one in \cite{ball:1474,Beck1,Beck2}.
  For this reason,  
  it is unclear how the results of this paper could be 
  derived directly from
  \cite{ball:1474,Beck1,Beck2} and whether such an approach would
  be simpler than the current one.

\textbf{Outline}
 In \S \ref{sect:switch} we review the definition of \BSLPV and the 
 related system-theoretic concepts. In \S \ref{lpv_real:pf}
 we establish the formal relationship between \BSLPV and 
 linear switched systems. In \S \ref{lpv_real}
 we present a Kalman-style realization theory for \SLPV. Finally,
 in \S \ref{sect:io_eq} we present the input-output equations
 describing the behavior of \SLPV.

\textbf{Notation}
  Denote by $\mathbb{N}$ the set of natural numbers including $0$.
 The notation described below is standard in automata theory, see \cite{AutoEilen}.
 Consider a (possibly infinite) set $X$. 
 Denote by $X^{+}$ the set of finite
 non-empty sequences of elements of $X$, i.e. 
 each $w \in X^{+}$ is
 of the form $w=a_{1}a_{2} \cdots a_{k}$,
  $a_1,a_2,\ldots,a_k \in X$, $k > 0$.
 The length of the sequence $w$  above is denoted by $|w|$. 
 We denote by $wv$ the concatenation of the sequences $w,v \in X^{+}$,
  i.e. if $w=a_1\cdots a_k$ and $v=v_1\cdots v_l$,
  $a_1,\ldots,a_k,v_1,\ldots,v_l \in X$, then 
  $wv=a_1\cdots a_kv_1\cdots v_l$.
We denote by $\epsilon$ the \emph{empty sequence}.
We define
$X^{*}=X^{+} \cup \{\epsilon\}$ as the set of all finite sequences
of elements of $X$, including the empty sequence.
By convention, $|\epsilon|=0$, and the concatenation is extended to
$X^{*}$ as follows: for all $w \in X^{*}$,
$w\epsilon=\epsilon w=w$.
For each $j=1,\ldots,m$, $e_j$ is the $j$th unit vector of $\mathbb{R}^{m}$, i.e.
  $e_j=(\delta_{1,j},\ldots, \delta_{n,j})$,
  $\delta_{i,j}$ is the Kronecker symbol. 
 If $Z$ is a subset of a vector space, then $\SPAN Z$ denotes
 the vector space spanned by the elements of $Z$.

\section{Discrete-time LPV systems}
\label{sect:switch}
 In this section we present the formal definition of \BSLPV along with a number of relevant system-theoretic concepts for \SLPV. 
 \begin{Definition}
  \label{switch:def}
  A discrete-time
  affine linear parameter-varying  system
  (abbreviated by \LPV) is of the form
  \begin{equation}
  \label{lin_switch0}
  \Sigma\left\{
  \begin{array}{lcl}
   x(t+1) &=& \sum_{q=1}^{\QNUM} (A_{q}x(t)+B_{q}u(t))p_q(t)  \\
   y(t)  &=& \sum_{q=1}^{\QNUM} (C_{q}x(t))p_q(t). 
  \end{array}\right.
  \end{equation}
  Here $\mathcal{P} \subseteq \mathbb{R}^{\QNUM}$ 
  is the space of scheduling parameters, 
  $\QNUM$ is a positive integer, 
  $p(t)=(p_1(t),\ldots,p_{\QNUM}(t)) \in \mathcal{P}$ is the
  scheduling signal, $u(t) \in \mathbb{R}^{m}$ is the input,
  $y(t) \in \mathbb{R}^{\OUTP}$ is the output and 
  $A_{q} \in \mathbb{R}^{n \times n}$,
  $B_{q} \in \mathbb{R}^{n \times m}$, $C_q \in \mathbb{R}^{\OUTP \times n}$, $q \in Q=\{1,\ldots,\QNUM\}$ are the system matrices.
  We will use the following short notation.
    \[ \SysLPV \] 
\end{Definition}
\begin{Notation}
 In the sequel, $Q=\{1,\ldots,\QNUM\}$.
\end{Notation}
 The definition above also allows for affine
 dependence on the scheduling parameters.
 To this end, choose $\mathcal{P}$ to be of the 
 form $\mathcal{P}=\{ (p_1,\ldots,p_{\QNUM}) \mid p_1=1, (p_2,\ldots,p_{\QNUM}) \in \hat{\mathcal{P}}\}$ for some set $\hat{\mathcal{P}} \subseteq \mathbb{R}^{\QNUM-1}$.
 Moreover, if 
the affine hull of $\hat{\mathcal{P}}$ equals $\mathbb{R}^{\QNUM-1}$,
then the linear span of $\mathcal{P}$ will be equal to
  $\mathbb{R}^{\QNUM}$. The latter property is important, because in
 the sequel we often use the technical assumption
 that $\mathcal{P}$ contains a basis of $\mathbb{R}^{\QNUM}$.
 
 Note that in our definition the output $y_t$ at time $t$ does not depend on the input
 at time $t$. This restriction is made in order to simplify notation and most of the results
 can be easily extended to include direct dependence of $y_t$ on $u_t$.

 Throughout the section, \emph{$\Sigma$ denotes an \LPV\ of the form \eqref{lin_switch0}.}
 The dynamics of $\Sigma$ is driven by \emph{the inputs $\{u(t)\}_{t=0}^{\infty}$ and the
scheduling parameters $\{p(t)\}_{t=0}^{\infty}$.} The state of the system
at time $t$ is $x(t)$.
 If $\mathcal{P}=\{e_1,\ldots,e_{\QNUM}\}$, where
 $e_i$ denotes the $i$th standard basis vector, $i=1,\ldots,\QNUM$,
 then the \BLPV $\Sigma$ can be viewed as a \emph{linear switched system}
 with the set of discrete modes being equal to $Q=\{1,\ldots,\QNUM\}$. 

In order to enable formal discussion, we define
a number of standard concepts such as input-output maps,
reachability, etc. for \SLPV.
\begin{Notation}[Generalized inputs]
 Denote $\HYBINP=\mathcal{P} \times \mathbb{R}^{m}$.
\end{Notation}
 We denote by  $\HYBINP^{*}$ (resp. $\HYBINP^{+}$) 
 the set of all 
  finite 
 (resp. non-empty and finite)
sequences
 of elements of $\HYBINP$.
 A sequence 
 \begin{equation}
 \label{inp_seq}
 w=(p(0),u(0))\cdots (p(t),u(t)) \in \HYBINP^{+} \mbox{, } t \ge 0 
 \end{equation}
  describes the scenario, when the scheduling parameter $p(i)$ and
  the input $u(i)$ are fed to $\Sigma$ at
 time $i$, for $i=0,\ldots,t$.
   \begin{Definition}[State and output]
     Let $x \in \mathbb{R}^{n}$ be a state of $\Sigma$.
     Define the \emph{input-to-state} map
     $x_{\Sigma,x}:\HYBINP^{+} \rightarrow \mathbb{R}^{n}$ 
     and \emph{input-output map} $y_{\Sigma,x}:\HYBINP^{+} \rightarrow \mathbb{R}^{\OUTP}$ of $\Sigma$ as follows.
     For any $w \in \HYBINP^{+}$ of the form (\ref{inp_seq}),
     define \( x_{\Sigma,x}(w) \) as
      the state $x(t)$ of $\Sigma$
     at time $t$, and define
     $y_{\Sigma,x}(w)$ as the output $y(t)$ of $\Sigma$
     at time $t$, if the initial state $x(0)$ of $\Sigma$ equals $x$,
     and the inputs $\{u(i)\}_{i=0}^t$ and
     the scheduling signal $\{p(i)\}_{i=0}^{t}$ are fed to $\Sigma$.
     Note that for $t=0$, $x_{\Sigma,x}(w)=x$.
   \end{Definition}
   The definition above implies that the potential
   input-output behavior
   of an \BLPV can be formalized as a map
   \begin{equation}
   \label{io_map}
     f:\HYBINP^{+} \rightarrow \mathbb{R}^{\OUTP}. 
   \end{equation}
   The value $f(w)$ for $w$ of the form \eqref{inp_seq} represents the output 
   of the underlying black-box system at time $t$, 
   if the inputs $\{u(i)\}_{i=0}^{t}$ and the
   scheduling parameters
   $\{p(i)\}_{i=0}^t$ are fed to the system.
   This black-box system may or may not admit a description by a \LPV.
   Next, we define
   when an \BLPV describes (realizes) $f$.
   \begin{Definition}[Realization]
   \label{switch_sys:real:def1}
    The \BLPV $\Sigma$ of the form \eqref{lin_switch0} 
    is a \emph{realization} of
    an input-output map $f$ of the form \eqref{io_map}, if
   $f$ equals the input-output map of $\Sigma$ which corresponds to the zero initial state, i.e. $f=y_{\Sigma,0}$.
   The map $y_{\Sigma,0}$ will be referred to as the
  \emph{input-output map of $\Sigma$} and it will be denoted by
  $y_{\Sigma}$.
   \end{Definition}
  Similarly to \cite{MP:BigArticlePartI, MP:RealForm},
  the results of this paper could be extended to
  families of input-output maps and multiple initial states.
   However, in order to keep the notation simple, 
   we deal  only with the case when the initial state is zero.

   \begin{Definition}[Input-output equivalence]
    Two \BSLPV\ $\Sigma_1$ and $\Sigma_2$ are said to be
    \emph{input-output equivalent}, if $y_{\Sigma_1}=y_{\Sigma_2}$.
  \end{Definition}

  \begin{Definition}[Reachability]
  \label{def:reach}
    Let $\Sigma$ be an \BLPV of the form \eqref{lin_switch0}.
    We say that $\Sigma$ is reachable, if the linear span
    of all the states of $\Sigma$ which are reachable from
    the zero initial state yields the whole space $\mathbb{R}^{n}$.
  \end{Definition}
   \begin{Definition}[Observability]
    The \LPV\  $\Sigma$ is called \emph{observable} if
    for any two states $x_1,x_2 \in \mathbb{R}^{n}$,
    \( 
      y_{\Sigma,x_1}=y_{\Sigma,x_2} \) implies $x_1=x_2$.
   \end{Definition}
   That is, observability means that if we pick any two distinct 
   states of the system, then for \textbf{some} input and scheduling signal, the resulting
  outputs will be different.

  Note that the concepts of reachability and observability
  presented above are strongly related to extended 
  controllability and observability matrices from 
  subspace identification of \BSLPV \cite{Verhaegen09}.
  Later on, we will show that the \BLPV is reachable
  if and only if the extended controllability matrix
  is full rank, and the \BLPV is observable if and only
  if the extended observability matrix is full rank.

  Finally, we recall the notion of isomorphism for
  \SLPV.
  \begin{Definition}[\BLPV isomorphism]
  \label{sect:problem_form:lin:morphism}
   Consider a \LPV\  $\Sigma_1$ of the form (\ref{lin_switch0}) and
   a \LPV\  $\Sigma_2$ of the form 
   \[ \Sigma_{2}=\SysLPV[a] \]
    with $n_a=n$.
    A nonsingular matrix
    $\MORPH \in \mathbb{R}^{n \times n}$
    is said to be an \emph{\LPV\  isomorphism}
    from $\Sigma_{1}$ to $\Sigma_{2}$,  if
   \begin{equation*}
     \forall q=1,\ldots,\QNUM: 
    A^{a}_{q}\MORPH=\MORPH A_{q}\mbox{,\ \ }  B_{q}^{a}=\MORPH B_{q}
    \mbox{,\ \ }
    C_{q}^{a}\MORPH =C_{q}.
  \end{equation*}
  \end{Definition}
  Note that in the definition of an \BLPV isomorphism, the state-space
  transformation $\MORPH$ does not depend on the scheduling parameter.
  Finally, below we define what we mean by the dimension
  minimality of a \LPV. 
   \begin{Definition}[Dimension]
   \label{switch_sys:dim:def}
    The dimension of $\Sigma$, denoted by $\dim \Sigma$, is 
    the dimension $n$ of its state-space.
   \end{Definition}
\begin{Definition}[Minimality]
 Let 
 $f$ be an
 input-output map.
 An \BLPV $\Sigma$ is \emph{a minimal realization of $f$}, if
 $\Sigma$ is a realization of $f$, and
 for any \LPV\ 
 $\hat{\Sigma}$ which is a realization 
 of $f$, $\dim \Sigma \le \dim \hat{\Sigma}$.
 We say that $\Sigma$ is \emph{minimal},  if $\Sigma$ is
 a minimal realization of its own input-output map $y_{\Sigma}$. 
\end{Definition}

\section{Relationship between linear switched systems and \SLPV}
\label{lpv_real:pf}
 In this section we establish a formal
 relationship between \BSLPV and linear switched systems.
  We start by stating the following assumption.
\begin{Assumption}
\label{assum1}
 In the rest of the paper, unless stated otherwise,
 we will assume that 
 the linear span of elements of $\mathcal{P}$ equals
 $\mathbb{R}^{\QNUM}$, i.e. $\mathcal{P}$ does not
 belong to any of the proper linear subspaces of 
 $\mathbb{R}^{\QNUM}$.
\end{Assumption}
 Note that the assumption above is not restrictive. Indeed,
 if $\mathcal{P}$ belongs to a $\hat{\QNUM}$ dimensional
 proper linear subspace $\mathcal{X}$ of $\mathbb{R}^{\QNUM}$, then 
 we can define a linear map
 $\mathbf{S}:\mathbb{R}^{\QNUM} \rightarrow \mathbb{R}^{\hat{\QNUM}}$ 
  such that $\mathbf{S}$ is injective on $\mathcal{X}$ 
  and replace the set of scheduling parameters
  by $\hat{\mathcal{P}}=\mathbf{S}(\mathcal{P})$. Since 
  $\mathbf{S}$ is linear, the parameters of the 
  resulting new LPV system will
  depend on the parameters in an affine way.

Next, we introduce the concept of \emph{generalized convolution representation} for input-output maps. This concept will
allow us to concentrate on input-output maps for which there
is a hope that they can be realized by \SLPV.
 \begin{Notation}
 \label{lpv:not1}
  Let $\underline{p}=p(0)\cdots p(t)$
  be a sequence of scheduling parameters and let
  $v=q_0\cdots q_t \in Q^{+}$, $q_0\cdots q_t \in Q$.
  Then $\underline{p}^{v}=p_{q_0}(0)p_{q_1}(1)\cdots p_{q_t}(t)$. 
 \end{Notation}
 \begin{Definition}[Convolution representation]
 \label{sect:io:def1}
  Let $f$ be an input-output map of the form \eqref{io_map}.
  The map $f$ has a \emph{generalized convolution
  representation (abbreviated as \GCR)}, if
  there exists a map
  $S^f:\{v \in Q^{+} \mid |v| > 1\} \rightarrow \mathbb{R}^{\OUTP \times m}$ such that
  for each $w \in \HYBINP^{+}$ of the form \eqref{inp_seq},
  \begin{equation}
  \label{sect:io:def1:eq1}
    \begin{split}
     & f(w)=
      \sum_{k=0}^{t-1} 
        \{\sum_{v \in Q^{+}, |v|=t-k+1} S^f(v)\underline{p}_{k:t}^{v}\}u(k), 
    \end{split}
  \end{equation}
   where
   $\underline{p}_{k:t}=p(k)p(k+1)\cdots p(t)$.
 \end{Definition}
  The convolution representation states that
  $f(w)$ is linear in control input and that 
  it is a homogeneous polynomial of degree one in 
  the scheduling parameters. The values of the
  map $S^f$ play the
  role of the coefficients of this polynomial.
  Note that the concept of \GCR\ above is a special case of
  \emph{impulse response representation (IRR)} in \cite{Toth1}.
  Note that since in the \BSLPV of interest the output at time $t$ does not
  depend on the input at time $t$, the summation  in \eqref{sect:io:def1:eq1}
  goes only up to $t-1$.
  Below we show that $S^f$ is uniquely determined by $f$ and that
  the existence of a \GCR\ implies
  that without loss of generality we can assume that
  $\mathcal{P}=\mathbb{R}^{\QNUM}$.
 \begin{Lemma}
 \label{lpv:lemma1}
  If $f$ has a \GCR, then the map $S^f$ is uniquely
  determined by $f$. Moreover, there exists a unique extension 
  $f_{ext}$ of
  $f$ to $\HYBINP^{+}_{ext}$, where 
  $\HYBINP_{ext}=(\mathbb{R}^{\QNUM} \times \mathbb{R}^m)$,
  such that $f_{ext}$ also admits a \GCR\ and
  $S^f=S^{f_{ext}}$.
 \end{Lemma}
 \begin{proof}[Proof of Lemma \ref{lpv:lemma1}]
  The fact that $f_{ext}$ exists relies on the
  fact that \eqref{sect:io:def1:eq1}
  is defined for any values of $p(0),\ldots,p(t) \in \mathbb{R}^{\QNUM}$, 
 and by noticing that the right-hand side of
 \eqref{sect:io:def1:eq1} is a sum of terms, each of which
  multilinear in $p(i),\ldots,p(t)$, $i=0,\ldots,t$.
  Recall that function $g(z_1,\ldots,z_k)$ is multi-linear,
  if for each $i=1,\ldots,k$, if we fix $z_1,\ldots,z_{i-1}, z_{i+1},\ldots,z_k$ and we vary only $z_i$, then
  $g$ is a linear function of $z_i$.
 Then set $f_{ext}$ as the value of the right-hand side
 of \eqref{sect:io:def1:eq1}.
 If the value of $f(w)$ is known for 
 $p(0),\ldots,p(t)$ where $p(0),\ldots,p(t)$ run 
 through a bases of $\mathbb{R}^{\QNUM}$, then these values
 uniquely determine the value of the right-hand side of
 \eqref{sect:io:def1:eq1}, and thus $f_{ext}$ exists and
 it is unique and $S^{f_{ext}}=S^f$.
 Finally, the uniqueness of $S^f$ follow by
 noticing that $S^f(v)u=f_{ext}((e_{q_0},u)(e_{q_1},0)\cdots (e_{q_t},0))$ for $v=q_0\cdots q_t$, $q_0,\ldots,q_t \in Q$,
  $u \in \mathbb{R}^{m}$.
\end{proof}
 In the sequel, we will restrict attention to input-output maps which admit a \GCR. This is not a strong restriction, since the input-output maps of \BSLPV always admits a \GCR. 
 \begin{Lemma}
 \label{lpv:col0}
 The \BLPV $\Sigma$ of the form \eqref{lin_switch0} 
 is a realization of an input-output map 
 $f$ if and only if $f$ has a  \GCR\ and 
 for all $v=q_0\cdots q_t \in Q^{+}$, $q_0,\ldots,q_t  \in Q$,
 $t > 0$
 \begin{equation} 
 \label{lpv:lemma-1:eq1}
   S^f(v)=C_{q_t}A_{q_{t-1}}A_{q_{t-2}}\cdots A_{q_1}B_{q_0}.
 \end{equation}
  If $t=1$, then $A_{q_{t-1}}A_{q_{t-2}}\cdots A_{q_1}$ is interpreted as the
  identity matrix.
 \end{Lemma} 

 Now we are ready to state the relationship between
\BSLPV and linear switched systems.  To this end, 
 we introduce the following notation.
\begin{Notation}[Switched generalized inputs]
 Denote $\mathcal{P}_{sw}=\{e_1,\ldots,e_{\QNUM}\}$ and
 $\HYBINP_{sw}=(\mathcal{P}_{sw} \times \mathbb{R}^{m})$.
\end{Notation}
 Recall that we can view linear switched systems
 as a subclass of \SLPV, such that the space of scheduling parameters equals $\mathcal{P}_{sw}$. 
Potential input-output maps of linear switched
  systems are maps of the form $f:\HYBINP^{+}_{sw} \rightarrow \mathbb{R}^{\OUTP}$ such that $f$ admits a \GCR. 
 Linear switched systems and their input-output maps
 in the sense of \cite{MPLBJH:Real} correspond to
 linear switched systems and their input-output maps in
 the above sense, if one identifies the scheduling parameter
 $e_q$ with the discrete mode $q \in Q$.
  We refer the reader
  to  \cite{MPLBJH:Real} for the notion of realization, minimality,
  observability, span-reachability, isomorphism. 
  Alternatively,
  all these notions are special cases of the corresponding 
  concepts for \SLPV,  
  if one identifies linear switched systems as a 
  subclass of \BSLPV.
  Note that the concept of span-reachability from \cite{MPLBJH:Real} corresponds to the concept of reachability as defined in Definition \ref{def:reach}.

  \begin{Definition}
   For each $f:\HYBINP^{+} \rightarrow \mathbb{R}^{\OUTP}$ admitting a \GCR, define the \emph{associated switched input-output
   map} $\SWI(f):\HYBINP^{+}_{sw} \rightarrow \mathbb{R}^{\OUTP}$
  as follows.
   Let 
  $f_{ext}$ be the extension of $f$ to $\HYBINP_{ext}^{+}$ as
  described in Lemma \ref{lpv:lemma1} and
  define $\SWI(f)$ as the restriction of $f_{ext}$ to 
  $\HYBINP_{sw}^{+} \subseteq \HYBINP^{+}_{ext}$.
  \end{Definition}
  By noticing that $S^f=S^{f_{ext}}=S^{\SWI(f)}$  we can
  in fact conclude that the correspondence between $f$ and
  $\SWI(f)$ is one-to-one.
   Next we will establish a correspondence between
   \BSLPV and linear switched systems.
  \begin{Definition}
   Let $\Sigma$ be a \BLPV of the form \eqref{lin_switch0}.
   Define the \emph{linear switched systems $\SWS(\Sigma)$
   associated with $\Sigma$} as the linear switched system
   $\SWS(\Sigma)=(\OUTP,m,n,\mathcal{P}_{sw},\{(A_q,B_q,C_q)\}_{q=1}^{\QNUM})$.
  \end{Definition}
  The following theorem collects the properties of
  the correspondence between linear switched systems and
  \SLPV.
  \begin{Theorem}
  \label{lpv:theo1}
  \begin{enumerate}
  \item An \BLPV $\Sigma$ is a realization of the input-output map
   $f$, if and only if $\SWS(\Sigma)$ is a realization of
   $\SWI(f)$.
  \item
      For any \BLPV $\Sigma$, $\dim \SWS(\Sigma)=\dim \Sigma$.
  \item
      Two \BSLPV $\Sigma_1$ and $\Sigma_2$ are isomorphic if and only
      if $\SWS(\Sigma_1)$ is isomorphic to
      $\SWS(\Sigma_2)$.
  \item 
      The \BLPV $\Sigma$ is reachable, observable, minimal if and only if $\SWS(\Sigma)$  is respectively reachable, observable, or minimal.
  \end{enumerate} 
  \end{Theorem}
  \begin{proof}[Sketch of the proof of Theorem \ref{lpv:theo1}]
   The only non-trivial statement is that $\SWS$ preserves
   reachability and observability. Let $\Sigma$ be an
   \BLPV of the form \eqref{lin_switch0}.
   First we show that $\Sigma$ is reachable 
   if and only if $\SWS(\Sigma)$ is reachable.
   To this end, consider the map input-to-state map
   $x_{\Sigma,0}:\HYBINP^{+} \rightarrow \mathbb{R}^{n}$
   of $\Sigma$. Notice that 
   $x_{\Sigma,0}$ can be extended to act on 
   $\HYBINP^{+}_{ext}$ and that for any
   input $w \in \HYBINP^{+}_{ext}$ of the form
   \eqref{inp_seq}, $x_{\Sigma,0}$ is a sum of terms, each of which is multilinear in 
   $p(0),\ldots,p(t)$. Hence, the linear span of
   the values of $x_{\Sigma,0}(w)$, $w \in \HYBINP_{ext}^{+}$ 
  equals the
   linear span of values of $x_{\Sigma,0}(w)$,
   $w \in (\mathbf{Z} \times \mathbb{R}^{m})^{+}$, where
   $\mathbf{Z}$ is a basis of $\mathbb{R}^{\QNUM}$. Since by
   Assumption \ref{assum1} $\mathcal{P}$ contains such a 
   basis of $\mathbb{R}^{\QNUM}$ and 
   $\mathcal{P}_{sw}$
   is a basis of $\mathbb{R}^{\QNUM}$,
   it follows that the linear span of $x_{\Sigma,0}(w)$,
   $w \in \HYBINP^{+}$ equals the linear span of
   $x_{\Sigma,0}(w)$, $w \in \HYBINP^{+}_{sw}$.
   Finally, notice that $x_{\Sigma,0}(w)=x_{\SWS(\Sigma),0}(w)$
   for all $w \in \HYBINP^{+}_{sw}$.
   Hence, $\Sigma$ is reachable if and only if
   $\SWS(\Sigma)$ is reachable.


   Next, we show that $\Sigma$ is observable if and only if
   $\SWS(\Sigma)$ is observable. To this end, notice
   that $y_{\Sigma,x}$ can be extended to $\HYBINP^{+}_{ext}$
   and that for any $w \in \HYBINP^{+}_{ext}$ of
   the form \eqref{inp_seq}, 
   $y_{\Sigma,x}(w)$ is a sum of terms, each of which is multilinear in $p(0),\ldots,p(t)$.
   Hence, $y_{\Sigma,x_1}$ and $y_{\Sigma,x_2}$ agree on
   $\HYBINP^{+}$, if they agree on any set
   $(\mathbf{Z} \times \mathbb{R}^{m})^{+}$, 
   where $\mathbf{Z}$ is a basis of $\mathbb{R}^{\QNUM}$.
   Since $\mathcal{P}_{sw}$ is a basis of $\mathbb{R}^{\QNUM}$
   and by Assumption \ref{assum1} 
   $\mathcal{P}$ contains a basis of $\mathbb{R}^{\QNUM}$,
   it then follows that $y_{\Sigma,x_1}$ and
   $y_{\Sigma,x_2}$ are equal on $\HYBINP^{+}$ if and only if
   they are equal on $\HYBINP^{+}_{sw}$. 
   Notice that for $i=1,2$,
   $y_{\SWS(\Sigma),x_i}$ coincides with the restriction of
   $y_{\Sigma,x_i}$ to the set $\HYBINP_{sw}^{+}$.
   This then implies
   that $\Sigma$ is observable if and only if $\SWS(\Sigma)$ is observable.
  \end{proof}

\section{Kalman-style realization theory}
\label{lpv_real}
In this section we exploit Section \ref{lpv_real:pf} and
realization theory of linear switched systems 
\cite{MP:RealForm,MP:BigArticlePartI,MP:PartReal,MPLBJH:Real}
to formulate a Kalman-style realization theory for
\SLPV. 

We start with presenting a characterization of
minimality.
\begin{Theorem}[Minimality]
\label{theo:min}
 An \BLPV is minimal, if and only if
 it is reachable and observable. 
 If two minimal \BSLPV are equivalent, then they are isomorphic.
\end{Theorem}
The theorem above is a direct consequence of Theorem \ref{lpv:theo1} and \cite[Theorem 3]{MPLBJH:Real}.

 Similarly to linear switched systems
 \cite{MPLBJH:Real}, one can construct example of an \BLPV
 $\Sigma$ which is minimal (reachable, observable), 
 while none of the linear subsystems
 $(A_q,B_q,C_q)$, $q \in Q$ is minimal (resp. reachable, observable). 

 Next, we present rank conditions for observability and reachability.
To this end, 
recall from \cite{RolandAbbas,Verhaegen09} the definition of
extended reachability and observability matrices for \SLPV.
That is, let $\Sigma$ be of the form \eqref{lin_switch0}.
We define the extended reachability matrices
$\mathcal{R}_i$, $i \in \mathbb{N}$ for $\Sigma$ as follows:
$\mathcal{R}_0=\begin{bmatrix} B_1, & B_2, & \ldots, & B_{\QNUM} \end{bmatrix}$ and for all $i \in \mathbb{N}$, let
\[ \mathcal{R}_{i+1}=\begin{bmatrix} R_i, & A_1\mathcal{R}_i, & A_2\mathcal{R}_i \ldots, A_{\QNUM}\mathcal{R}_i \end{bmatrix}
\]
Similarly, we define the extended observability matrices
$\mathcal{O}_i$ for $\Sigma$ recursively as follows:
$\mathcal{O}_0=\begin{bmatrix} C_1^T, & C_2^T, & \ldots, & C_{\QNUM}^T \end{bmatrix}^T$ and for all $i \in \mathbb{N}$,
\[ \mathcal{O}_{i+1}=\begin{bmatrix} O_i^T, 
    A_1^T\mathcal{O}_{i}^T, & A_2^T\mathcal{O}_{i}^T, & \ldots, & A_{\QNUM}^T\mathcal{O}^T_i \end{bmatrix}^T.
 \]
 Notice that $\mathcal{R}_{n-1}$ equals the 
 reachability matrix of the switched system $\SWS(\Sigma)$
 and
 $\mathcal{O}_{n-1}$ equals the observability matrix of
 $\SWS(\Sigma)$. For the definition of reachability and observability matrices for linear switched systems see \cite{MPLBJH:Real}.
 Hence, Theorem \ref{lpv:theo1} and \cite[Theorem 4]{MPLBJH:Real} 
 yield the following rank conditions.
\begin{Theorem}
\label{sect:real:lemma1}
 The \BLPV $\Sigma$ is reachable if and only if
 $\Rank \mathcal{R}_{n-1}=n$, and $\Sigma$ is observable if 
 and only if $\Rank \mathcal{O}_{n-1}=n$.
\end{Theorem}
Theorem \ref{sect:real:lemma1} yields 
algorithms for reachability, observability and minimality reduction of \SLPV. These algorithms are the same as those for
linear switched systems \cite{MPLBJH:Real}.
 Next, we present the necessary and sufficient conditions for
 the existence of a \LPV\ realization  for an input-output map.
 To this end, we need the notion of the Hankel-matrix  and
 Markov-parameters of  an input-output map. 
 In the sequel, \emph{$f$ denotes a map of the form \eqref{io_map}, 
 and we assume that $f$ has a \GCR.}
\begin{Definition}[Markov-parameters]
\label{sect:io:def0}
 The  \emph{Markov-parameter} $M^f(v)$ of $f$ indexed by 
 the sequence
 $v \in Q^{*}$ is the following
 $\OUTP\QNUM \times \QNUM m$ matrix 
 \begin{equation}
   \label{main_results:lin:arb:pow2}
     M^f(v) = \begin{bmatrix}
               S^f(1v1), & \cdots, & S^f(\QNUM v1) \\
               S^f(1v2), & \cdots, & S^f(\QNUM v2) \\
             \vdots   & \vdots     & \vdots \\
             S^f(1v\QNUM), & \cdots, & S^f(\QNUM v \QNUM)
             \end{bmatrix}.
\end{equation}
\end{Definition}
 That is, $M^{f}(v)$ can be viewed as a $\QNUM \times \QNUM$
 block matrix, such that the
 $(i,j)$th entry of $M^f(v)$ equals $S^f(jvi)$, $j,i \in Q$.

  If $f$ has an \BLPV realization $\Sigma$, then
  from Lemma \ref{lpv:col0} it follows that
  $M^f(v)$ can be expressed as product of matrices of $\Sigma$:
  if $\Sigma$ is as in \eqref{lin_switch0}, then
  $M^f(\epsilon)=\widetilde{C}\widetilde{B}$ and
  for all $v=q_1,\ldots,q_k \in Q$, $k > 0$,
  \begin{equation}
  \label{sect:io:lemma1:eq1}
	\begin{split}
       M^f(v)=\widetilde{C}A_{q_k}A_{q_{k-1}}\cdots A_{q_1}\widetilde{B},
      \end{split}
  \end{equation}
  where $\widetilde{C}=\begin{bmatrix} C_1^T, & & \ldots, & C_{\QNUM}^T \end{bmatrix}^T$,
  $\widetilde{B}=\begin{bmatrix} B_1, & \ldots, B_{\QNUM} \end{bmatrix}$.

Note that the values of the map $S^f$, and hence
the Markov-parameters $\{M^f(v)\}_{v \in Q^{*}}$ can be
obtained from the values of $f$.
A naive way to compute $S^f$ is
to compute the derivatives of $f$ with respect to the
scheduling parameter. 
It is easy to see that the Markov-parameters $f$ and 
$\SWI(f)$ coincide, i.e. $M^f(v)=M^{\SWI(f)}(v)$, $v \in Q^{*}$.  Moreover, when applied to linear switched systems,
the Markov-parameters from 
Definition \ref{sect:io:def0} coincide
with the ones in
\cite[Definition 12]{MPLBJH:Real}.

Note that the definition of Markov-parameters does not assume the
existence of an \BLPV realization of $f$.
In fact, even if $f$ does not admit a finite dimensional 
state realization,
its Markov-parameters remain well-defined. 
The reason for this choice is that we want to use the Markov-parameters
to characterize the existence of a  
finite dimensional \BLPV realization of $f$.
This will be achieved by constructing a Hankel-matrix from the
Markov-parameters and by proving that $f$ has an \BLPV
realization if and only if the rank of that Hankel-matrix is finite.
Of course, for this to make sense, we have to define the Markov-parameters and the Hankel-matrix as objects which are well-defined even in
the absence of a finite dimensional state-space representation.

 In order to define the Hankel-matrix of $f$, 
 we will introduce a lexicographic
 ordering on the set $Q^{*}$.
\begin{Definition}[Lexicographic ordering]
\label{rem:lex}
 Recall that $Q=\{1,\ldots,\QNUM\}$. We define a lexicographic
 ordering $\prec$ on $Q^{*}$ as follows.
 For any $v,s \in Q^{*}$,
 $v \prec s$ holds if either
 \textbf{(a)}
 $|v| < |s|$, or 
 \textbf{(b)} $0 < |v|=|s|=k$, $v \ne s$ and the following
 holds: 
 $v=q_1\cdots q_k$, $s=s_1\cdots s_k$, $q_1,\ldots,q_k, s_1,\ldots,s_k \in Q$, and
  for some $l \in \{1,\ldots, k\}$,
 $q_l < s_l$ with the usual ordering of integers and
 $q_i=s_i$ for $i=1,\ldots, l-1$.
 Note 
 that $\prec$ is a complete ordering and
 \begin{equation}
 \label{rem:lex:eq1}
  Q^{*}=\{v_1,v_2,\ldots \} 
 \end{equation}
 with $v_1 \prec v_2 \prec \ldots $.
 Note that $v_1 =\epsilon$ and for all $i \in \mathbb{N}$, $q \in Q$,
 $v_i \prec v_iq$.
\end{Definition}
\begin{Definition}[Hankel-matrix] 
\label{main_result:lin:hank:arb:def}
 Define the Hankel-matrix $H_f$ of $f$ as the following
 infinite matrix
 \[ H_f = \begin{bmatrix}
     M^f(v_1v_1), & M^f(v_2v_1), & \cdots, & M^f(v_kv_1), & \cdots \\
     M^f(v_1v_2), & M^f(v_2v_2), & \cdots, & M^f(v_{k}v_2), & \cdots \\
     M^f(v_1v_3)
     & M^f(v_2v_3), & \cdots, & M^f(v_{k}v_3), & \cdots \\
     \vdots   
      & \vdots   & \cdots & \vdots & \cdots 
   \end{bmatrix},
 \]
i.e. the $\OUTP\QNUM \times m\QNUM$ block of $H_f$ 
in the block row $i$ and block column $j$
equals the Markov-parameter $M^f(s)$, where the word $s=v_jv_i \in Q^{*}$
is the concatenation of the words $v_j$ and $v_i$ from 
\eqref{rem:lex:eq1}.
\end{Definition}
 Note that $H^{f}=H^{\SWI(f)}$ and the definition of
 the Hankel-matrix coincides with the one for linear switched systems
 \cite[Definition 13]{MPLBJH:Real}. 
  \begin{Theorem}[Main result on existence]
  \label{sect:real:theo2}
      The map $f$ has a realization by  an \LPV\  if and
      only if $f$ has a \GCR\ 
      and 
     $\Rank H_{f} < +\infty$.
     Any minimal \LPV\ 
     realization of $f$ has dimension equal to $\Rank H_f$.
\end{Theorem}
The theorem above is a direct consequence of Theorem \ref{lpv:theo1} and \cite[Theorem 5]{MPLBJH:Real}.

 Finally, we prove the correctness of the
 Kalman-Ho-like realization algorithm for \BSLPV
 from \cite{RolandAbbas}. A similar algorithm was formulated
 for linear switched systems in \cite{MP:PartReal,MPLBJH:Real}.
 To this end, we need the following definition.
 For every $L \in \mathbb{N}$, 
  denote by $\mathbf{N}(L)=\sum_{j=0}^{L} \DNUM^j$ the
 number such all the sequences $v \in Q^{*}$ of length at most $L$.
 Due to the properties of lexicographic ordering, it follows that
 $\{v_1,\ldots,v_{\mathbf{N}(L)}\}=\{v \in Q^{*} \mid |v| \le L\}$.
 \begin{Definition}
 Denote by $H_{f,L,M}$ the $\mathbf{N}(L)\OUTP\QNUM \times \mathbf{N}(M)m\QNUM$ upper-left sub-matrix of $H_f$.
 \end{Definition}
 If $f$ is realized by an \BLPV $\Sigma$, 
 then $H_{f,L,M}=\mathcal{O}_{L}\mathcal{R}_M$,
 where $\mathcal{O}_L$ is the $K$th extended observability matrix
 and $\mathcal{R}_M$ is the $M$th extended reachability
 matrix of $\Sigma$. In this case $H_{f,L,M}$ coincides
 with the Hankel-matrix defined in \cite{RolandAbbas}.
 The Kalman-Ho algorithm goes as follows.
 Compute the factorization
 \[ H_{f,L,L+1} = \mathbf{O}\mathbf{R} \]
  such $\mathbf{O} \in \mathbb{R}^{\OUTP\QNUM\mathbf{N}(L) \times n}$,
  $\mathbf{R} \in \mathbb{R}^{n \times m\QNUM\mathbf{N}(L+1)}$
  and $\Rank \mathbf{O}=\Rank \mathbf{R}=n$ for $n=\Rank H_{f,L,L+1}$.
        One way to compute this factorization is by
   SVD decomposition as in \cite{RolandAbbas}, 
  i.e.  if $H_{f,L,L+1}=USV^{T}$ is the SVD
  decomposition of $H_{f,L,L+1}$ where $S$ is the diagonal part,
  then set $\mathbf{O}=US^{1/2}$ and $\mathbf{R}=S^{1/2}V^T$.
  Let $\overline{\mathbf{R}}$ be the matrix formed by the
  first $\mathbf{N}(L)m\QNUM$ columns of $\mathbf{R}$.
  For each $q \in Q$, let $\mathbf{R}_{q}$ be the 
  $n \times \mathbf{N}(L)m\QNUM$ matrix, such that 
  the $j$th $n \times m\QNUM$ block column of $\mathbf{R}_{q}$
  equals to the $k$th  $n \times m\QNUM$ block column of 
  $\mathbf{R}$, where $k$ is such that $v_jq=v_k$. Here
  $v_k$ and $v_j$ are the $j$th and $k$th elements of
  the lexicographic ordering \eqref{rem:lex:eq1}.
  Construct $\Sigma$ of the form (\ref{lin_switch0}) such that $\begin{bmatrix} B_1, & \ldots, & B_{\QNUM}\end{bmatrix}$
 equals the first $m\QNUM$ columns of $\mathbf{R}$,
 $\begin{bmatrix} C_1^T, & C_2^T, & \ldots, & C_{\QNUM}^T \end{bmatrix}^T$
 equals the first $\OUTP\QNUM$ rows of $\mathbf{O}$ and
 $A_q=\mathbf{R}_{q}\overline{\mathbf{R}}^{+}$,
      where $\overline{\mathbf{R}}^{+}$ is the Moore-Penrose pseudoinverse
      of $\overline{\mathbf{R}}$.
  \begin{Theorem}
  \label{part_real_lin:theo1}
  If $\Rank H_{f,L,L}=\Rank H_{f}$, then 
  $\Sigma$ computed by the algorithm above 
  is a minimal realization of $f$.
  The condition $\Rank H_{f,L,L}=\Rank H_{f}$ holds, if
  there exists an \LPV\  realization $\Sigma$ of $f$ such that
  $\dim \Sigma \le L+1$.
\end{Theorem}
The theorem above is a direct consequence of Theorem \ref{lpv:theo1} and \cite[Theorem 6]{MPLBJH:Real}.

\section{Input-output equations for \BSLPV}
\label{sect:io_eq}
 In this section we use the results of realization theory to
 establish a relationship between \BSLPV
 and input-output equations.
 In the sequel, \emph{$f$ is assumed to be an input-output map
 $f:\HYBINP^{+} \rightarrow \mathbb{R}^{\OUTP}$ and it is assumed that
 $f$ admits a \GCR}.
In order to avoid excessive notation, in this section we
assume that $\OUTP=1$. However, all the results can easily be
extended to several outputs.
\begin{Definition}[Input-output equations]
 \label{lpv:io_def:theo1}
  An \emph{affine polynomial equation} $E(\mathbf{P},\mathbf{Y},\mathbf{U})$ of order $n$ is a
  polynomial
  in variables $\mathbf{P}=\{P_{i,j}\}_{i=0,\ldots,n,j \in Q}$,
  $\mathbf{Y}=\{Y_i\}_{i=0}^{n}$, $\mathbf{U}=\{U_{i,j}\}_{i=1,\ldots,n, j=1,\ldots,m}$ such that
  \begin{equation}
  \label{sect:io_eq:eq1}
  \begin{split}
     E(\mathbf{P},\mathbf{Y},\mathbf{U})= \sum_{j=0}^{n} Q_{j}(\mathbf{P})Y_j+
     \sum_{i=1}^{n} \sum_{j=1}^{m} L_{i,j}(\mathbf{P})U_{i,j}
  \end{split}
  \end{equation}
  where $Q_{0}(\mathbf{P})$, $Q_i(\mathbf{P})$, 
  $L_{i,j}(\mathbf{P})$ are
  polynomials, $i=1,\ldots,n$, $j=1,\ldots,m$ and
  $Q_{0}(\mathbf{P}) \ne 0$.
\end{Definition}
\begin{Definition}
  Assume that $E$ is an affine polynomial
  equation of the form \eqref{sect:io_eq:eq1}.
  Then the input-output map $f$
  is said to \emph{satisfy the equation $E$}, if
  for each $w$ of the form \eqref{inp_seq} wit $t > n$,
  \( E(f,w)=0 \), where $E(f,w)$ denotes the 
   value of $E(\mathbf{P},\mathbf{Y},\mathbf{U})$ with
  the following substitution 
  $P_{i,j}=p_{j}(t-i)$,  $U_{i,l}=u_l(t-i)$
  $Y_i=f((p(0),u(0))\cdots (p(t-i),u(t-i))$ for
  $j \in Q$, $l=1,\ldots,m$, $i=0,\ldots,n$.
\end{Definition}
 \begin{Theorem}
 \label{lpv:io_eq:theo1}
  Assume that the set of scheduling parameters $\mathcal{P}$
  is an open subset of $\mathbb{R}^{\QNUM}$.
  The input-output map $f$ has a realization by an \BLPV
  if and only if $f$ satisfies an affine polynomial
  equation of the form \eqref{sect:io_eq:eq1}.
 \end{Theorem}
  In \cite{Toth1} it was shown that input-output maps of LPV 
  systems
  with a meromorphic dependence on parameters correspond to
  input-output maps which satisfy 
  linear autoregressive equations with
  respect to outputs and inputs. The coefficients of these
  autoregressive equations were meromorphic functions of the
  time-shifted scheduling parameters. Affine polynomial
  input-output equations represent a special case of the
  autoregressive equations of \cite{Toth1}. 
  Theorem \ref{lpv:io_eq:theo1} says that input-output maps
  described by these type of equations (and which, in addition, admit a \GCR) correspond precisely to input-output maps realizable by \SLPV.

 The proof of Theorem \ref{lpv:io_eq:theo1}
 is an adaptation of the proof of the analogous statement for state-affine systems
 \cite{Son:Resp,Son:Real}.
 The proof is divided into several lemmas,
 proofs of which are presented in the appendix.

 \begin{Lemma}
 \label{lpv:io_eq:lemma-1}
  If the interior of $\mathcal{P}$ not empty, then
  $f$ satisfies the input-output
  equation \eqref{sect:io_eq:eq1} if and only if
  its extension $f_{ext}$ from Lemma \ref{lpv:lemma1}
  satisfies \eqref{sect:io_eq:eq1}.
 \end{Lemma}
 From Lemma \ref{lpv:io_eq:lemma-1} it follows that
 without loss of generality, 
 we can assume $\mathcal{P}=\mathbb{R}^{\QNUM}$.
 \begin{Assumption}
 In the sequel, \emph{we assume that $\mathcal{P}=\mathbb{R}^{\QNUM}$}.
 \end{Assumption}

 For any
 sequence $\underline{p}=p_1p_2\cdots p_k \in \mathcal{P}^{+}$, 
 $p_1,\ldots,p_k \in \mathcal{P}$, $k > 0$ define the map
 $f^{\underline{p}}:\HYBINP^{+} \rightarrow \mathbb{R}$ 
 as follows:
 \[ \forall w \in \HYBINP^{+}: f^{\underline{p}}(w)=f(w(p_1,0)(p_2,0)\cdots (p_k,0))
 \]
 Recall that $w(p_1,0)\cdots (p_k,0)$ denotes the concatenation
 of the sequence $w$ with the sequence $(p_1,0)\cdots (p_k,0)$.
 Intuitively, $f^{\underline{p}}(w)$ equals the response of $f$, if
 first we feed in the 
 inputs and scheduling parameters prescribed by $w$ and
 then for the last $k$ time steps we feed in the zero input and 
 the scheduling parameters $p_1,\ldots,p_k$.
  
  \begin{Lemma}
  \label{lpv:io_eq:lemma1}
   There exists  an affine polynomial input-output equation
  $E$ of the form \eqref{sect:io_eq:eq1}  such that
  $f$ satisfies $E$, 
  if and only if there exists polynomials
  $Q_i(\mathbf{P})$, $i=0,\ldots,n$ such that $Q_0 \ne 0$,
  and for any $p_1,\ldots,p_{n+1} \in \mathcal{P}$,
  \begin{equation}
  \label{sect:io_eq:eq2}
     \sum_{j=0}^{n} Q_j(p_1,\ldots,p_{n+1})
      f^{p_{1}p_{2}\cdots,p_{n+1-j}}
  \end{equation}
  \end{Lemma} 
  Before formulating the next statement, recall
  the set of all maps
  $g:\HYBINP^{+} \rightarrow \mathbb{R}^{\OUTP}$
  forms a vector space with respect to
  point-wise addition and multiplication by scalar.
 \begin{Lemma}
 \label{lpv:io_seq:lemma4}
  The map $f$ satisfies \eqref{sect:io_eq:eq2} for some
  $Q_j$, $j=0,\ldots,n$ if and only if 
  $\mathcal{W}_{f}=\SPAN\{f^{\underline{p}} \mid \underline{p} \in (\mathbb{R}^{\QNUM})^{+} \}$ is finite dimensional.
 \end{Lemma}
 \begin{Lemma}
 \label{lpv:io_seq:lemma5}
  The input-output map $f$ has a realization by a \BLPV
  if and only if 
  $\mathcal{W}_f=\SPAN\{f^{\underline{p}} \mid \underline{p} \in (\mathbb{R}^{\QNUM})^{+} \}$ is finite dimensional.
 \end{Lemma}
 The proof of Lemma \ref{lpv:io_seq:lemma5} boils down to showing 
 that there is a linear isomorphism between $\mathcal{W}_f$
 and 
 the linear space spanned by the rows of the Hankel-matrix $H_f$ of $f$.
 Hence, $\mathcal{W}_f$ is finite dimensional if and only if
 $\Rank H_f < +\infty$.
 By Theorem \ref{sect:real:theo2}, the latter is equivalent to 
 the existence of an \BLPV realization of $f$.
 Theorem \ref{lpv:io_eq:theo1} follows from the
 lemmas above as follows. 
 From Lemma \ref{lpv:io_seq:lemma5}, 
 $f$ has a realization by a 
 \BLPV if and only if 
 $\mathcal{W}_f$ is finite dimensional. By Lemma 
 \ref{lpv:io_seq:lemma4} and Lemma \ref{lpv:io_eq:lemma1},
 the latter is equivalent to existence of an affine
 polynomial equation of the form
 \eqref{sect:io_eq:eq1} such that $f$ satisfies $E$.

\section{Conclusion}
We have presented realization theory for the class of affine 
LPV systems.
In addition, we have shown that realization theory of this class of
LPV systems is equivalent to that of for linear switched systems.
We have also presented an equivalent input-output representation for
affine LPV systems.



\appendix
\section{Technical proofs}
\label{app:proofs}
 \begin{proof}[Proof of Lemma \ref{lpv:io_eq:lemma-1}]
  Note that $E(f_{ext},w)$ is a polynomial
  expression in 
  $p(0),u(0)), \ldots, (p(t),u(t)) \in \HYBINP_{ext}$ for each $w$ of the form \eqref{inp_seq}.
 Since $E(f_{ext},w)=0$ for all $w \in \HYBINP^{+}$
 and $\HYBINP$ is an open subset of $\HYBINP_{ext}$, it then
 follows that $E(f_{ext},w)=0$ for all $w \in \HYBINP_{ext}$.
 \end{proof}
  \begin{proof}[Proof of Lemma \ref{lpv:io_eq:lemma1}]
   By substitution it is clear that if \eqref{sect:io_eq:eq1} holds,
   then \eqref{sect:io_eq:eq2} holds. 
   Conversely assume that \eqref{sect:io_eq:eq2} holds.
   With the notation of \eqref{sect:io_eq:eq1}, notice that
   for $Y_{i}=f((p(0),u(0))\cdots (p(t-i),u(t-i))$, $i=0,\ldots,n$,
   \[ 
     \begin{split}
     & Y_i=f^{p(t-n)p(t-n+1) \cdots, p(t-i)}(v)+\\
     & + \sum_{j=i+1}^{n}R_{j,i}(p(t-j),\ldots,p(t-i))u(t-j),
    \end{split}
   \] 
    where $v=(p(0),u(0))\cdots (p(t-n-1),u(t-n-1))$,
    and $R_{j,i}$, $j=n,\ldots,i+1$ are suitable polynomials.
   Consider the expression $\sum_{i=0}^{n} Q_i(p(t-n),\ldots,p(t))Y_i$.
   By grouping together the terms $(Q_iR_{j,i})(p(t-n),\ldots,p(t))u(t-j)$
   in a suitable way, we can obtain polynomials $L_{j,l}$, $j=1,\ldots,n$
  and $l=1,\ldots,m$ such that
  then \eqref{sect:io_eq:eq1} holds.
  \end{proof}
 \begin{proof}[Proof of Lemma \ref{lpv:io_seq:lemma4}]
Assume that $f$ satisfies \eqref{sect:io_eq:eq2}. 
  It then follows from 
  $f^{p_1,\ldots, p_{r+i}}(w)=f^{p_{k+1},\ldots,p_{r+i}}(w(p_1,0)\cdots(p_r,0))$,
  $i=1,\ldots,n+1$,
  $w \in \HYBINP^{+}$  that
   for any $p_1,\ldots, p_{r+n+1} \in \mathcal{P}$, $r \ge 0$,
  \begin{equation}
  \label{sect:io_eq:eq22}
  \sum_{j=0}^{n} f^{p_1,\ldots,p_{n-j+r+1}}Q_j(p_{r+1},\ldots,p_{r+n+1})=0.
  \end{equation}
  Define $W_{f,k}=\mathrm{span}\{f^{p_1\cdots p_k} \mid p_1,\ldots,p_k \in \mathcal{P}\}$.
  Note that as $\mathcal{P}$ contains an open set, there exists a finite
  subset $E$ of $\mathcal{P}$ such that the elements of $E$ span $\mathbb{R}^{\QNUM}$.
  Since $f^{p_1,\ldots,p_k}$ is multi-linear in $p_1,\ldots,p_k$,
  $f^{p_1,\ldots,p_k}$ is a linear combination of 
  $f^{v_1,\ldots,v_k}$, $(v_1,\ldots,v_k) \in E^k$. Since $E^k$ is a finite set.
  it then follows that $W_{f,k}$ is finite dimensional.

Since
  $Q_{0} \ne 0$, there exists an open and dense subset
  $Z \subseteq \mathcal{P}^{r+n+1}$ such that for any
  $(p_1,\ldots,p_{n+r+1}) \in Z$, $Q_{0}(p_r,\ldots,p_{n+r+1}) \ne 0$,
  $r=1,\ldots,k$.
  By dividing \eqref{sect:io_eq:eq22} with $Q_{0}(p_r,\ldots,p_{n+r+1}) \ne 0$,
  and using  induction on $r$,
  it then follows that $f^{p_1 \cdots p_{n+r+1}} \in W_{f,n}$.
  Since $Z$ is dense, it then follows that for any $\underline{p}=p_1\cdots p_{r+n+1} \in \mathcal{P}^{+}$,
   there exist $\underline{p}_{j}=p_1^j\cdots p_{r+n+1}^j$, $p_1^j,\ldots,p_{r+n+1}^j \in Z$, $j \in \mathbb{N}$, such that
   $\lim_{j \rightarrow \infty} p^j_i = p_i$, $i=1,\ldots,r+n+1$.
   As $f^{\underline{p}}$ is polynomial (multi-linear) in the entries of
 $p_1,\ldots,p_{n+r+1}$, it then follows
 $\lim_{j \rightarrow \infty} f^{\underline{p}_j}(w)=f^{\underline{p}}(w)$ for all $w \in \HYBINP^{+}$, i.e.
 $f^{\underline{p}_j} \in W_{f,n}$ converges to $f^{\underline{p}}$ point-wise. Since
 $W_{f,n}$ is finite dimensional vector space, it then follows that
 $f^{\underline{p}} \in W_{f,n}$.

 Indeed, let $g_1,\ldots,g_K$ be a basis of $W_{f,n}$. 
 Then there exist $w_1,\ldots,w_K \in \HYBINP^{+}$ and
  such that the matrix
  $S=(S_{r,l})_{r,l=1,\ldots,K}$, $S_{r,l}=g_{l}(w_r) \in \mathbb{R}$, is invertible.
  Define $v_j=(f^{\underline{p}_j}(w_1), \ldots, f^{\underline{p}_j}(w_K))^T$,
  and 
  $v=(f^{\underline{p}}(w_1), \ldots, f^{\underline{p}}(w_K))^T$,
  It then follows that $f^{\underline{p}_j}=\sum_{k=1}^{K} \alpha_k^j g_k$,
  where $\alpha^j=(\alpha_1^j,\ldots,\alpha_K^j)^T$ satisfies $\alpha_j=S^{-1}v_j$.
  We claim that if $(\alpha_1,\ldots,\alpha_K)^T=\alpha = S^{-1}v$, then
  $f^{\underline{p}}=\sum_{k=1}^{K} \alpha_k  g_k$.
  Assume the contrary. Then for some $w \in \HYBINP^{+}$,
  $f^{\underline{p}}(w) \ne \sum_{k=1}^{K} \alpha_k g_k(w)$.
   Notice that $\lim_{j \rightarrow \infty} v_j =v$ and hence
   $\alpha = S^{-1}v = \lim_{j \rightarrow \infty} S^{-1}v_j=\lim_{j \rightarrow \infty} \alpha^j$.
  Hence, $\sum_{k=1}^{K} \alpha_k g_k(w)=\lim_{j \rightarrow \infty}\sum_{k=1}^{K} \alpha_k^j g_k(w)=\lim_{j \rightarrow \infty} f^{\underline{p}_j}(w)=f^{\underline{p}}(w)$, which is a contradiction.

 Hence, $f^{\underline{p}} \in W_{f,n}$, for all $\underline{p} \in \mathcal{P}^{+}$,
 $|\underline{p}|=n+r+1$,  and thus $W_{f,n}=W_f$, i.e. $W_f$ is finite dimensional.

 Conversely, assume that $\mathcal{W}_f$ is finite dimensional.
 For each $v=q_1\cdots q_k \in Q^{+}$, $q_1,\ldots,q_k \in Q$, denote by $f^{v}$ the map $f^{e_{q_1}\cdots,e_{q_k}}$.
 As it was noted above,
 $f^{p_1 \cdots p_k}$ is multilinear in $p_1,\ldots,p_k$ and
 hence $\mathcal{W}_f$
 equals the linear span of 
 $f^{z_1},\ldots, f^{z_d}$ for some
 $z_1,\ldots,z_d \in Q^{+}$.
 Notice that for any $\underline{p} \in (\mathbb{R}^{\QNUM})^{+}$,
 $|\underline{p}|=k$,
 $f^{\underline{p}}=\sum_{v \in Q^{+}, |v|=k} f^{v} \underline{p}^{v}$.
  Since for every $v \in Q^{+}$, $f^v$ is a linear combination of
  $f^{z_i}$, $i=1,\ldots,d$, there exist
  polynomials $P_{i,k}$ is $k\QNUM$ variables, such that
  $f^{p_1\cdots p_k}=\sum_{j=1}^{d} P_{j,k}(p_1,\ldots,p_k)f^{z_j}$ for any $p_1,\ldots,p_k \in \mathbb{R}^{\QNUM}$.

  Consider now the $d \times (d+1)$ polynomial matrix 
  $\mathbf{D}_{d+1}$ in variables 
  $X_i=(X_{i,1},\ldots,X_{i,\QNUM})$, $i=1,2,\ldots,d+1$
  such that
  $(i,j)$the entry of $\mathbf{D}_{d+1}$ equals
  $P_{i,j}(X_1,\ldots,X_j)$, 
  Let's view $\mathbf{D}_{d+1}$
  as a matrix with elements in
  $\mathbb{R}(X_1,\ldots,X_{d+1})$. 
  Here, $\mathbb{R}(X_1,\ldots,X_{d+1})$ is the quotient field of the
  polynomial ring $\mathbb{R}[X_1,\ldots,X_{d+1}]$. 
  Since $\mathbf{D}_{d+1}$ has only $d$ rows and $d+1$ columns,
  the columns of $\mathbf{D}_{d+1}$ must be linearly dependent.
  It then follows that
  there exist polynomials $D_j,N_j \in \mathbb{R}[X_1,\ldots,X_{d+1}]$,
  $N_j \ne 0$, $j=1,\ldots,k^{*}$,
  such that $D_{k*} \ne 0$ and
  \( \sum_{j=1}^{k^{*}} P_{i,j}\frac{D_j}{N_j}=0 \).
  By multiplying the equation above by the product of $N_1\cdots N_{k^{*}}$ we get that
  \begin{equation} 
  \label{lpv:io_seq:lemma4:eq2}
    \forall i = 1,\ldots, d: 
    \sum_{j=1}^{k^{*}} P_{i,j}R_{j} = 0 
  \end{equation}
  for some polynomial $R_1,\ldots R_{k^{*}}$, $R_{k^{*}} \ne 0$.
  Notice that the polynomial $P_{i,j}$ depend only on the variables
  $X_1,\ldots,X_j$, hence $R_1,\ldots, R_{k^{*}}$ can be
  chosen to be polynomials only in $X_1,\ldots, X_{k^{*}}$.
  If $k^{*}=1$, then $P_{i,1}=0$ and hence $f^{p}=0$ for all
  $p \in \mathcal{P}$. Hence, $f^{p_1,\ldots,p_k}(w)=f^{p_k}(w(p_1,0)\cdots (p_{k-1},0))=0$ for 
  for all $w \in \HYBINP^{+}$, $p_1,\ldots,p_k$, $k > 0$. Then \eqref{sect:io_eq:eq2} holds
  for $n=1$ with any choice of $Q_1$ and $Q_0$.
  If $k^{*} > 1$, then set $n=k^{*}$,
  $Q_{i}=R_{k^{*}-i}$, $i=1,\ldots k^{*}-1$.
   Using the fact that $f^{p_1\cdots p_i}=\sum_{j=1}^{d} P_{j,i}(p_1,\ldots,p_i)f^{z_j}$ and \eqref{lpv:io_seq:lemma4:eq2}, 
   it then follows that
  \eqref{sect:io_eq:eq2} holds for all $p_1,\ldots,p_k \in \mathbb{R}^{\QNUM}$.
 \end{proof}
 \begin{proof}[Proof of Lemma \ref{lpv:io_seq:lemma5}]
  Denote by $\mathcal{H}$ the linear span of the rows of
  the Hankel-matrix $H_f$. Notice that each element of
  $\mathcal{H}$ can be viewed as a sequence of
  $1 \times \QNUM m$ matrices.
  We define the linear map
  $\Phi: \mathcal{W}_f \rightarrow \mathcal{H}$ as follows:
  $\Phi(f^{\underline{p}})=(H_{v_1},H_{v_2},\ldots)$, such that
  for each $v \in Q^{*}$,
  \[
     H_{v}=\sum_{s \in Q^{+}, |s|=|\underline{p}|-1} 
         \begin{bmatrix} \underline{p}^{s1} &,\ldots,\underline{p}^{s\QNUM} \end{bmatrix} M^f(vs).
  \]
  In other words, $H_v=\begin{bmatrix} H_{v,1} &,\ldots, & H_{v,\QNUM} \end{bmatrix}$, where
  $H_{v,q}=\sum_{s \in Q^{+}, |s|=|\underline{p}|} S^f(qvs)\underline{p}^{s}$.
  Moreover, for any $w$ of form \eqref{inp_seq}, 
  \[ f^{\underline{p}}(w)=\sum_{k=0}^{t-1}H_{q_{k+1}\cdots q_t,q_k}u(k)p_{q_1}(k)\cdots p_{q_t}(t).
  \]
  Hence, it is clear that $\Phi$ is an injective linear map.
  Moreover, the row of $H_f$ indexed by the integer
  $l=(i-1)\QNUM+q$, $q \in Q$, $i=1,\ldots$ equals
  $\Phi(f^{e_q})$ if $i=1$, or
  $\Phi(f^{e_{q_1}\cdots e_{q_k},e_q})$, if $i > 1$
   and $q_1,\ldots,q_k \in Q$ are such that
   $v_i=q_1\cdots q_k$, where $v_i$ is $i$th sequence of the 
   lexicographic ordering \eqref{rem:lex:eq1}.
  Hence, $\Phi$ is a linear isomorphism from $\mathcal{W}_f$ to
  the space spanned by the rows of $H_f$. The rest of
  follows from Theorem \ref{sect:real:theo2}.
 \end{proof}

\end{document}